\documentclass[12pt,a4paper]{amsart}

\usepackage{amsmath,amssymb,amsfonts}
\usepackage[utf8]{inputenc}
\usepackage[T1]{fontenc}
\usepackage{newtxtext,newtxmath,mathtools}
\usepackage[colorlinks=true,linkcolor=teal,citecolor=magenta,urlcolor=blue]{hyperref}
\usepackage[margin=20mm,top=28mm]{geometry}
\usepackage{enumerate}

\newtheorem{theorem}{Theorem}[section]
\newtheorem{corollary}[theorem]{Corollary}
\newtheorem{lemma}[theorem]{Lemma}
\newtheorem{proposition}[theorem]{Proposition}
\theoremstyle{definition}
\newtheorem{example}[theorem]{Example}
\newtheorem{definition}[theorem]{Definition}
\newtheorem{remark}[theorem]{Remark}

\newcommand{\K}{\Bbbk}
\newcommand{\Hilb}{\mathrm{Hilb}}
\newcommand{\FP}{\mathrm{FP}}
\newcommand{\Tor}{\operatorname{Tor}}
\newcommand{\Hom}{\operatorname{Hom}}
\newcommand{\GKdim}{\operatorname{GKdim}}
\newcommand{\Ext}{\operatorname{Ext}}
\newcommand{\ord}{\operatorname{ord}}
\newcommand{\length}{\operatorname{length}}
\newcommand{\pole}{\operatorname{pole}}
\newcommand{\Nil}{\operatorname{nil}}

\renewcommand{\leq}{\leqslant}
\renewcommand{\geq}{\geqslant}
\numberwithin{equation}{section}
\title{Signed Euler--Smith Profiles of Graded Algebras}

\author{Atabey Kaygun}
\address{Istanbul Technical University\\
Istanbul, Turkey}
\email{kaygun@itu.edu.tr}

\subjclass[2020]{Primary 16P90; Secondary 16W50, 15A21, 16D90}
\keywords{Euler--Smith profiles, matrix Hilbert series, Smith normal form, Gelfand--Kirillov dimension.}

\begin{document}

\begin{abstract}
Matrix Hilbert series retain local elementary-divisor data that their determinants
discard.  For locally finite positively graded elementary algebras whose vertex
simples are of type $\FP_\infty$ and whose matrix Hilbert series is rational, we
construct from possibly infinite minimal resolutions a rational Euler matrix and
its signed local Smith profile at $x=1$.  The profile refines pole-order growth
for perfect modules and is invariant under shift-compatible graded Morita and
perfect-derived equivalences.  Every finite integer profile occurs, while a
wholly negative profile forces exponential corner growth.  This extends the
nonnegative local Smith framework for twisted Calabi--Yau algebras to a setting
in which the Euler matrix may itself have poles.
\end{abstract}

\maketitle

\section{Introduction}

For a graded algebra with several primitive idempotents, the matrix Hilbert
series records directional growth that is lost both in the total Hilbert series
and in its determinant.  Reyes--Rogalski made this distinction effective in the
rational case: the nonzero poles of the matrix entries determine the
finite-growth versus exponential-growth branch, and in the finite branch the
largest pole order at $x=1$ is the Gelfand--Kirillov dimension
\cite[Lemma~2.7 and Proposition~2.8]{ReyesRogalski2019}.  Their
Example~2.9 also shows why the determinant is too coarse: cancellation can make
the multiplicity of $1$ in a determinantal denominator strictly larger than the
actual growth degree.

In the homologically smooth elementary case, the same work identifies the
denominator matrix with homological data.  Finite projective resolutions of the
vertex simples produce a polynomial matrix $q_A(x)$ with
$C_A(x)=q_A(x)^{-1}$ \cite[Proposition~4.2]{ReyesRogalski2019}; the later
structure theory of Reyes--Rogalski places this construction in the setting of
locally finite graded twisted Calabi--Yau and generalized Artin--Schelter
regular algebras \cite{ReyesRogalski2022}.  The present paper starts from this
homological denominator but removes the finite-resolution hypothesis.

The local Smith viewpoint itself was developed in
\cite[Section~2.1]{kaygun2026localsmithprofiles} for twisted Calabi--Yau
algebras.  There the homological denominator is polynomial and the local
exponents at $x=1$ are nonnegative; the additional Calabi--Yau identities then
impose Nakayama and parity constraints on that profile.  Here the purpose is
different.  We retain the same local elementary-divisor philosophy but allow
possibly infinite $\FP_\infty$ resolutions and a rational Euler matrix with
zeros or poles at $x=1$.  The resulting relative position is therefore signed.
The standard transport of local Smith data under changes of the local
coordinate, recalled in \cite[Section~1.1]{kaygun2026localsmithprofiles}, is
also the only formal input needed for the Koszul-duality discussion below.

Let
\[
 A=\bigoplus_{n\geq0}A_n,\qquad
 A_0=\bigoplus_{i=1}^r\K e_i,
\]
be locally finite and positively graded.  We assume that the vertex simples are
of graded type $\FP_\infty$ and that the matrix Hilbert series $C_A(x)$ is
rational.  Positive grading makes the alternating Euler sums of the minimal
resolutions coefficientwise finite.  Theorem~\ref{thm:completed-euler-cartan}
therefore constructs a completed Euler matrix, and
Proposition~\ref{prop:rational-algebraization} identifies its rational
algebraization with
\[
 \mathcal E_A(x)=C_A(x)^{-1}\in GL_r(\mathbb Q(x)).
\]
This is the first step: it extends the polynomial homological denominator of the
smooth case to the $\FP_\infty$ setting without changing its formal relation to
the matrix Hilbert series.

At the growth point $x=1$, the lattices
\[
 \mathbb Q\llbracket x-1\rrbracket^r,\qquad
 \mathcal E_A(x)\mathbb Q\llbracket x-1\rrbracket^r
\]
have a unique relative Smith position
$a_1\leq\cdots\leq a_r$.  We call this the signed Euler--Smith profile.
Its positive and negative elementary divisors are encoded by
$\mathcal G_A^+$ and $\mathcal G_A^-$.  The determinant retains only the signed
total length $\sum_i a_i$, whereas the pair of modules retains the complete
local elementary-divisor data.  This is precisely the matrix information that
is invisible to determinantal pole counting.

Two consequences give the profile its content.  First, realization is
unrestricted: Theorem~\ref{thm:realization} shows that every nondecreasing
integer tuple occurs.  The nontrivial realization problem is therefore not
existence itself, but connected realization modulo a common shift; explicit
tensor and square-zero constructions give large families in that direction.
Second, the signs interact sharply with growth.  For every perfect graded
module, Proposition~\ref{prop:local-growth-class} identifies its pole order at
$1$ with the torsion order of the corresponding class in $\mathcal G_A^+$.
Theorem~\ref{thm:perfect-global-local-dichotomy} separates this local degree from
the global choice of growth branch, and Corollary~\ref{cor:euler-smith-growth}
gives, in the finite branch,
\[
 \GKdim A=\max\{0,a_r\}.
\]
Conversely, Proposition~\ref{prop:negative-top} shows that $a_r<0$ forces
exponential growth in every diagonal corner.  Negative exponents are therefore
freely realizable but cannot occur wholly in the finite-growth regime.

The profile is also categorical once the grading translation is retained.
Theorem~\ref{thm:morita-invariance} proves invariance under shift-compatible
graded Morita equivalence.  Lemma~\ref{lem:euler-form} identifies $C_A$ as the
Gram matrix of the sesquilinear Euler form on graded $K$-theory, and
Theorem~\ref{thm:derived-invariance} upgrades the invariance to
shift-compatible equivalences of perfect derived categories.  The Zhang-twist
example shows that compatibility with the grading shift is essential.

The organization follows these dependencies.  Section~\ref{sec:euler-inversion}
constructs and algebraizes the Euler matrix.  Section~\ref{sec:local-smith}
introduces the signed local Smith data.  Section~\ref{sec:realization} proves
complete realization, develops connected constructions, and records the
Koszul-dual interpretation at the antipodal point.  Section~\ref{sec:growth}
then develops module and algebra growth and isolates the exponential consequence
of wholly negative profiles.  Section~\ref{sec:morita} proves categorical
invariance and its sharpness.

\subsection*{Use of large language models}

Large language models were used for copy editing and to help identify references
and exact computations requiring verification.  They were not treated as
mathematical authorities: all arguments, references, and computations were checked
independently, and the author is responsible for the mathematical content.

\section{Completed Euler inversion and rational algebraization}
\label{sec:euler-inversion}

Throughout, $\K$ is the arbitrary base field over which the algebras,
modules, and categorical equivalences are defined; no characteristic assumption
is imposed.  Graded dimensions are ordinary integers, so the matrix Hilbert
series and completed Euler matrix below have integral coefficients even when
$\operatorname{char}\K>0$.  Rationality below therefore refers to rational
functions over $\mathbb Q$, and the local Smith calculation at $x=1$ is performed
over $\mathbb Q\llbracket x-1\rrbracket$.

A graded module is \emph{perfect} if it has a finite resolution by finitely
generated graded projectives, and is of type $\FP_\infty$ if it has a
projective resolution by finitely generated graded projectives in every
homological degree.  We use $M\langle m\rangle_n=M_{n-m}$ for grading shifts.

Let
\[
  A=\bigoplus_{n\geq0}A_n,
  \qquad
  A_0=E=\bigoplus_{i=1}^r\K e_i,
  \qquad
  J=A_{>0},
\]
where $A$ is locally finite and positively graded, and $e_1,\ldots,e_r$ are
primitive orthogonal idempotents.  Write $S_j=Ae_j/Je_j$ for the vertex
simples.  For a locally finite graded $E$-bimodule $V$, set
\[
  \Hilb_V(x)_{ij}:=\sum_{m\geq0}\dim_\K(e_iV_me_j)x^m.
\]
For a graded left module $M$, let $\Hilb_M(x)$ be the column vector whose $i$th
entry is $\Hilb_x(e_iM)$.  Regarding $A$ as a graded $E$-bimodule, write
$C_A(x):=\Hilb_A(x)$ for its matrix Hilbert series, so that
$C_A(x)_{ij}=\Hilb_x(e_iAe_j)$.  In particular,
$C_A(x)\in M_r(\mathbb Z\llbracket x\rrbracket)$.

Assume throughout this section that every vertex simple $S_j=Ae_j/Je_j$ is of
type $\FP_\infty$.  We use the standard
minimal-resolution theory for bounded-below graded modules over a locally finite
graded algebra with finite-dimensional degree-zero part.  Existence of minimal
graded projective resolutions is recalled in
\cite[Section~2, paragraph preceding Lemma~2.2]{ReyesRogalski2022}, and the
graded Nakayama lemma used below is
\cite[Lemma~2.2]{ReyesRogalski2022}.

\subsection{Completed inversion and rationality}

\begin{lemma}[Graded projective decomposition]
\label{lem:graded-projective-decomposition}
Every finitely generated graded projective left $A$-module $P$ admits a finite
decomposition
\[
  P\cong\bigoplus_{i=1}^r Ae_i\otimes_\K W_i,
\]
where $W_i\cong e_i(P/JP)$ is finite-dimensional and graded.
\end{lemma}

\begin{proof}
For each $i$, choose homogeneous lifts in $e_iP$ of a basis of $e_i(P/JP)$.  They define a
graded map
\[
  \phi:\bigoplus_iAe_i\otimes_\K e_i(P/JP)\longrightarrow P
\]
which is an isomorphism modulo $J$.  Graded Nakayama
\cite[Lemma~2.2]{ReyesRogalski2022} gives surjectivity.  Since $P$ is
projective, $\phi$ splits; its finitely generated kernel satisfies $K=JK$, hence
vanishes by the same lemma.
\end{proof}

Choose a minimal graded resolution
\[
\cdots\longrightarrow P_2^{(j)}\longrightarrow P_1^{(j)}
  \longrightarrow P_0^{(j)}\longrightarrow S_j\longrightarrow0
\]
with
\[
P_p^{(j)}\cong\bigoplus_{i=1}^r Ae_i\otimes_\K W_{p,ij}.
\]
Minimality means that every differential has image in $JP_{p-1}^{(j)}$;
this is the form in which minimality enters the proof of
\cite[Lemma~2.6]{ReyesRogalski2022}, and it is equivalent to the
projective-cover definition recalled before
\cite[Lemma~2.2]{ReyesRogalski2022}.  Since $J$ is
positively graded and $S_j$ is concentrated in degree zero, induction on $p$
shows that every minimal generator of $P_p^{(j)}$ has degree at least $p$.
Thus $(W_{p,ij})_m=0$ for $m<p$, so only finitely many homological degrees
contribute to a fixed internal degree.  Applying $E\otimes_A-$ kills the
differentials and gives
\[
e_i\Tor_p^A(E,S_j)\cong W_{p,ij}.
\]

\begin{definition}[Completed Euler matrix]
Define
\[
\widehat{\mathcal E}_A(x)_{ij}
  :=\sum_{p\geq0}(-1)^p\Hilb_x(W_{p,ij})
  \in\mathbb Z\llbracket x\rrbracket.
\]
\end{definition}

\begin{theorem}[Completed Euler--Cartan inversion]
\label{thm:completed-euler-cartan}
One has $C_A(x)\widehat{\mathcal E}_A(x)=I_r$ in $M_r(\mathbb Z\llbracket x\rrbracket)$.
\end{theorem}

\begin{proof}
Fix an internal degree.  The corresponding strand of a minimal resolution has
only finitely many nonzero terms, so its ordinary Euler characteristic gives
$\Hilb_{S_j}(x)=C_A(x)\widehat{\mathcal E}_A(x)e_j$ coefficientwise.  Since
$\Hilb_{S_j}=e_j$, assembling the columns proves the identity.
\end{proof}

\begin{definition}[Rational algebraization]
The completed Euler matrix is \emph{rationally algebraizable} if it is the
germ at $x=0$ of a matrix
$\mathcal E_A(x)\in GL_r(\mathbb Q(x))$.  We say that $C_A(x)$ is
\emph{rational} when its entries lie in $\mathbb Q(x)$.
\end{definition}

The algebraization, when it exists, is unique because each rational-function
entry is determined by its formal germ at a point where it is regular.

\begin{proposition}[Rationality criterion]
\label{prop:rational-algebraization}
Under graded $\FP_\infty$, the completed Euler matrix is rationally
algebraizable if and only if the matrix Hilbert series $C_A(x)$ is rational.
\end{proposition}

\begin{proof}
Theorem~\ref{thm:completed-euler-cartan} identifies the two formal matrices as
mutual inverses.  If $C_A\in M_r(\mathbb Q(x))$, then $C_A(0)=I_r$ implies that
its inverse over $\mathbb Q(x)$ has the required formal germ.  Conversely, if
$\mathcal E_A\in GL_r(\mathbb Q(x))$, then its inverse is a rational matrix
whose formal germ is $C_A$.
\end{proof}

Whenever these conditions hold we write
$C_A(x)=\mathcal E_A(x)^{-1}$ as an identity in $M_r(\mathbb Q(x))$.

\begin{remark}[Scope of the standing hypotheses]
\label{rem:scope-rationality}
If $A$ is left noetherian, then the vertex simples are automatically of graded
type $\FP_\infty$; rationality of $C_A$ remains a separate hypothesis.  A
standard source of rationality is the Hilbert--Serre theorem: if $A$ is finite
as a graded module over a finitely generated connected commutative positively
graded central subalgebra, then every corner $e_iAe_j$ has rational Hilbert
series \cite[p.~401, first paragraph]{AvramovBuchweitzSally1997}.  Noetherianity
is not assumed in the theory below.
\end{remark}

\subsection{Scope beyond finite projective dimension}

\begin{example}[A nonsmooth $\FP_\infty$ family]
\label{ex:nonsmooth-fpinfty}
Let
\[
  D_d=\K[y_1,\ldots,y_d,z]/(z^2),
  \qquad \deg y_i=\deg z=1,
  \qquad d\geq1.
\]
Over $D=\K[z]/(z^2)$ the periodic complex
\[
  \cdots\xrightarrow{z}D\langle3\rangle\xrightarrow{z}D\langle2\rangle
  \xrightarrow{z}D\langle1\rangle\xrightarrow{z}D
  \longrightarrow\K\longrightarrow0
\]
is a minimal exact resolution.  Tensoring it with the finite Koszul resolution
over $\K[y_1,\ldots,y_d]$ shows that the augmentation module is of graded type
$\FP_\infty$ over $D_d$ but has infinite projective dimension.  Thus $D_d$ is
not homologically smooth; since $(D_d)_0=\K$ is separable, compare
\cite[Theorem~3.10]{ReyesRogalski2022}.

Nevertheless
\[
  C_{D_d}(x)=\frac{1+x}{(1-x)^d},
  \qquad
  \mathcal E_{D_d}(x)=\frac{(1-x)^d}{1+x}.
\]
Hence the completed Euler matrix algebraizes rationally; moreover it has a zero
of order $d$ at $x=1$.  The point here is the scope of the homological
construction: completed Euler inversion remains available although the vertex
simple has infinite projective dimension.
\end{example}

Since $\mathcal E_A(0)=I_r$, the local behavior at the origin is trivial.  The
core theory therefore begins at $x=1$, the point distinguished by growth.

\section{Signed local Smith data}
\label{sec:local-smith}

The completed inversion theorem supplies the rational matrix to which local
Smith theory can now be applied.  From now on the \emph{standing hypotheses}
are that every vertex simple is
of graded type $\FP_\infty$ and that $C_A(x)\in M_r(\mathbb Q(x))$, equivalently
that the completed Euler matrix is rationally algebraizable by
Proposition~\ref{prop:rational-algebraization}.  Unless explicitly stated
otherwise, every algebra considered from this point onward is understood to
satisfy these hypotheses.  When a construction below introduces a new algebra,
part of the assertion is that the construction remains within this class; we
verify this in the proof without repeating the standing hypotheses in each
statement.

Set $t=x-1$.  For a finite-length $\mathbb Q\llbracket t\rrbracket$-module $T$, set
$\Nil_t(T)=\min\{d\geq0:t^dT=0\}$, with $\Nil_t(0)=0$, and use
$\ord_t(0)=+\infty$.  Define the Euler lattice
\[
\Lambda_A:=\mathcal E_A(1+t)\,\mathbb Q\llbracket t\rrbracket^r
\subset \mathbb Q((t))^r.
\]

There exist $U,V\in GL_r(\mathbb Q\llbracket t\rrbracket)$ and unique integers $a_1\leq\cdots\leq a_r$ such that
\begin{equation}
\label{eq:local-smith-form}
  U\,\mathcal E_A(1+t)\,V=\operatorname{diag}(t^{a_1},\ldots,t^{a_r}).
\end{equation}
Indeed, multiply $\mathcal E_A(1+t)$ by a large power of $t$, apply ordinary
Smith normal form over the DVR $\mathbb Q\llbracket t\rrbracket$, and subtract
that power from the invariant exponents.  For the
corresponding local construction for matrix polynomials, compare
\cite[Theorems~1 and~3]{WilkeningYu2011}.  We call
$(a_1,\ldots,a_r)$ the \emph{signed Euler--Smith profile at $1$}.

Define
\[
\mathcal G_A^+:=\mathbb Q\llbracket t\rrbracket^r/
  (\mathbb Q\llbracket t\rrbracket^r\cap\Lambda_A),
  \qquad
  \mathcal G_A^-:=\Lambda_A/
  (\mathbb Q\llbracket t\rrbracket^r\cap\Lambda_A).
\]
These finite-length $\mathbb Q\llbracket t\rrbracket$-modules are the positive and
negative Euler--Smith modules at $1$.

The Smith form immediately identifies the positive and negative quotients:
\[
  \mathcal G_A^+\cong\bigoplus_{a_i>0}\mathbb Q\llbracket t\rrbracket/(t^{a_i}),
  \qquad
  \mathcal G_A^-\cong\bigoplus_{a_i<0}\mathbb Q\llbracket t\rrbracket/(t^{-a_i}).
\]
Thus the two modules are simply the positive and negative elementary divisors
of the relative lattice position.

For $1\leq s\leq r$, let $\delta_s(\mathcal E_A)$ be the least $t$-adic
valuation of a nonzero $s\times s$ minor of $\mathcal E_A(1+t)$, with $\delta_0=0$.  The
Smith form gives
\begin{equation}
  \delta_s(\mathcal E_A)=a_1+\cdots+a_s,
\end{equation}
and in particular
\begin{equation}
  \ord_{x=1}\det\mathcal E_A(x)
  =\sum_i a_i
  =\length_{\mathbb Q\llbracket t\rrbracket}\mathcal G_A^+
   -\length_{\mathbb Q\llbracket t\rrbracket}\mathcal G_A^-.
\end{equation}
Thus the determinant records only the signed total length, whereas the pair
$(\mathcal G_A^+,\mathcal G_A^-)$ retains the complete signed Smith profile.
For polynomial Euler matrices the corresponding nonnegative construction is
studied in \cite[Section~2.1]{kaygun2026localsmithprofiles}; here
$\mathcal E_A(1+t)$ may lie anywhere in $GL_r(\mathbb Q((t)))$, so the relative
position is genuinely signed.  Both the determinantal order formula and the Smith
form itself apply verbatim to any matrix in $GL_r(\mathbb Q((t)))$; in
particular the exponents of $C_A(1+t)=\mathcal E_A(1+t)^{-1}$ are
$(-a_r\leq\cdots\leq-a_1)$.  The same local Smith construction can be made at any closed point by passing
to its completed local DVR.  We work at $x=1$ because it is the point relevant
to growth; Subsection~\ref{sec:koszul-duality} uses the same construction at
$x=-1$.

The local data is gauge invariant in the following sense, which we use
repeatedly.

\begin{lemma}[Gauge invariance]
\label{lem:gauge}
Let $H,H'\in GL_r(\mathbb Q((t)))$ with $H'=UHV$ for some
$U,V\in GL_r(\mathbb Q\llbracket t\rrbracket)$.  Then $H$ and $H'$ have the
same exponent multisets, and multiplication by $U$ carries the lattice pair
$(\mathbb Q\llbracket t\rrbracket^r,H\mathbb Q\llbracket t\rrbracket^r)$ to
$(\mathbb Q\llbracket t\rrbracket^r,H'\mathbb Q\llbracket t\rrbracket^r)$,
inducing
isomorphisms of the associated modules $\mathcal G^{\pm}$.
\end{lemma}

\begin{proof}
Smith forms of $H$ and $H'$ differ by units, and
$H'\mathbb Q\llbracket t\rrbracket^r
=UHV\mathbb Q\llbracket t\rrbracket^r
=U(H\mathbb Q\llbracket t\rrbracket^r)$, while
$U\mathbb Q\llbracket t\rrbracket^r=\mathbb Q\llbracket t\rrbracket^r$.
\end{proof}

The nonnegative case is exactly the case $\mathcal G_A^-=0$, equivalently
$\mathcal E_A$ is regular at $x=1$.  We do not impose this restriction: the
negative part is essential to the realization and growth results below.

\section{Realization and duality}
\label{sec:realization}

Once the signed profile has been defined, its unrestricted realization can be
settled before any growth theory is used.  Three elementary operations suffice.
For products,
\[
 C_{A\times B}=C_A\oplus C_B,\qquad
 \mathcal E_{A\times B}=\mathcal E_A\oplus\mathcal E_B,
\]
so profiles concatenate and
$\mathcal G^\pm_{A\times B}\cong\mathcal G^\pm_A\oplus\mathcal G^\pm_B$.
The standing hypotheses are inherited factorwise.

If $R_0=\K$ has one vertex and profile $(c)$, then
\[
 C_{R\otimes B}(x)=C_R(x)C_B(x),\qquad
 \mathcal E_{R\otimes B}(x)=\mathcal E_R(x)\mathcal E_B(x),
\]
and the profile of $R\otimes B$ is
$(b_1+c,\ldots,b_r+c)$ when the profile of $B$ is
$(b_1,\ldots,b_r)$.  Indeed,
$\mathcal E_R(1+t)=t^c u(t)$ for a unit
$u(t)\in\mathbb Q\llbracket t\rrbracket^\times$, and tensor products of
graded $\FP_\infty$ resolutions preserve the standing homological
hypothesis.  Polynomial stabilization is the special case
$R=\K[y_1,\ldots,y_d]$:
\[
 C_{B\otimes\K[y_1,\ldots,y_d]}=(1-x)^{-d}C_B,\qquad
 \mathcal E_{B\otimes\K[y_1,\ldots,y_d]}=(1-x)^d\mathcal E_B,
\]
so every exponent increases by $d$.

\subsection{A negative seed and complete realization}

\begin{proposition}[A negative rank-one seed]
\label{prop:negative-seed}
Let
\[
 R=\K\langle u,v,w\rangle/(wu,wv,w^2,uw,vw),
 \qquad \deg u=\deg v=\deg w=1.
\]
Then $R$ satisfies the standing hypotheses and has signed profile $(-1)$.
More precisely,
\[
 C_R(x)=\frac1{1-2x}+x
       =\frac{(1-x)(1+2x)}{1-2x},
 \qquad
 \mathcal E_R(x)=\frac{1-2x}{(1-x)(1+2x)},
\]
so
$\mathcal G_R^-\cong\mathbb Q\llbracket t\rrbracket/(t)$ and
$\mathcal G_R^+=0$.
\end{proposition}

\begin{proof}
As a graded vector space, $R$ is the direct sum of the free algebra
$\K\langle u,v\rangle$ and the one-dimensional space $\K w$ in degree
one.  This gives the displayed Hilbert series.  At $x=1+t$,
\[
 \mathcal E_R(1+t)
 =t^{-1}\frac{1+2t}{3+2t},
\]
so its unique local Smith exponent is $-1$.

It remains to check the homological hypothesis.  Put $L=Rw=\K w$ and
$Q=R/L$.  If $J=R_{>0}$, then multiplication by $u$, $v$, and $w$ gives
\[
 J=Ru\oplus Rv\oplus Rw
   \cong Q\langle1\rangle^{\oplus2}\oplus\K\langle1\rangle,
\]
while the quotient map $R\to Q$ has kernel
$L\cong\K\langle1\rangle$.  Hence
\[
 \Omega_R(\K)\cong Q\langle1\rangle^{\oplus2}\oplus\K\langle1\rangle,
 \qquad
 \Omega_R(Q)\cong\K\langle1\rangle.
\]
Induction therefore expresses every syzygy of $\K$ as a finite direct
sum of shifts of $\K$ and $Q$; taking projective covers gives a graded
$\FP_\infty$ resolution.  Thus the unique vertex simple is of type
$\FP_\infty$, and rationality was already established by the displayed
formula for $C_R$.
\end{proof}

\begin{theorem}[Complete realization]
\label{thm:realization}
Every nondecreasing tuple of integers
\[
 a_1\leq\cdots\leq a_r
\]
occurs as the signed Euler--Smith profile of a locally finite positively
graded elementary algebra satisfying the standing hypotheses.
\end{theorem}

\begin{proof}
For $a\geq0$, the polynomial algebra
$P_a=\K[z_1,\ldots,z_a]$, with $P_0=\K$, has
$\mathcal E_{P_a}(x)=(1-x)^a$ and profile $(a)$.  For $a<0$,
Proposition~\ref{prop:negative-seed} and the tensor-shift observation above show that
$R^{\otimes(-a)}$ has profile $(a)$.  Choose in this way a one-vertex
algebra $R_{a_i}$ of profile $(a_i)$ for each $i$.  Then
\[
  \prod_{i=1}^r R_{a_i}
\]
has profile $(a_1,\ldots,a_r)$ by the product formula above.
\end{proof}

Thus there is no sign obstruction to the profile itself.  The residual
realization problem with structure is to prescribe a profile by a connected,
hence ring-indecomposable, algebra, naturally modulo the common shifts above.

\subsection{Connected realizations}

For connected realization it is useful to keep two elementary algebra
constructions explicit.  If $N(x)\in M_r(\mathbb Z_{\geq0}[x])$ has
$N(0)=0$ and $V$ is a graded $E$-bimodule with
$\Hilb_V(x)=N(x)$, then the tensor algebra $T_E(V)$ has
\[
 \mathcal E_{T_E(V)}(x)=I_r-N(x).
\]
Indeed, multiplication identifies $T_E(V)\otimes_E V$ with the augmentation
ideal, so
\[
 0\longrightarrow T_E(V)\otimes_E V\longrightarrow T_E(V)
 \longrightarrow E\longrightarrow0
\]
gives finite projective resolutions of the vertex simples.

Dually, for the square-zero algebra $A_N=E\oplus V$ with $V^2=0$,
\[
 C_{A_N}(x)=I_r+N(x),\qquad
 \mathcal E_{A_N}(x)=(I_r+N(x))^{-1}.
\]
Since $J=V$ and $JV=0$, the first syzygy of each vertex simple is a finite
sum of shifted vertex simples, so the standing $\FP_\infty$ hypothesis is
preserved.  If
$0=b_1\leq\cdots\leq b_r$ are the local Smith exponents of
$I_r+N(1+t)$, inversion gives the signed profile
$(-b_r,\ldots,-b_1)$.  These two constructions prescribe, respectively, the
Euler matrix and its inverse.

The square-zero construction already gives useful connected normalized
families.  For every $m\geq1$ there is a connected two-vertex example with
profile $(-m,0)$.  Put $f=(1-x^2)^m$ and write
$f=1+n_+(x)-n_-(x)$ by separating the positive and negative coefficients
of $f-1$.  Since the coefficient of $x$ in $f$ vanishes,
$n_-(x)/x$ is a polynomial with nonnegative coefficients and zero constant
term.  Set
\[
  N(x)=\begin{pmatrix}n_+(x)&n_-(x)/x\\ x&0\end{pmatrix}.
\]
Then $N(0)=0$, all coefficients are nonnegative, and
$\det(I_2+N)=f$.  Since $(I_2+N)_{22}=1$, the local Smith exponents of
$I_2+N(1+t)$ are $(0,m)$, so the square-zero formula gives
profile $(-m,0)$.  The coefficient of $x$ in $n_-(x)/x$ is $m$, hence
the two vertices are joined in both directions in the square-zero quiver.

For every $k\geq1$ there is likewise a connected $(k+1)$-vertex example
with profile $(\underbrace{-1,\ldots,-1}_{k},0)$.  Take
$N=x(J-I)$ in size $k+1$, with $J$ the all-ones matrix.
The symmetric matrix $J-I$ has eigenvalues $k$, once, and $-1$, with
multiplicity $k$, so over $\mathbb Q$ the matrix $I+(1+t)(J-I)$ is
conjugate, by a constant invertible matrix, to
$\operatorname{diag}(k+1+kt,-t,\ldots,-t)$.  Its local Smith exponents
are $(0,\underbrace{1,\ldots,1}_{k})$, so the resulting square-zero
algebra has profile $(\underbrace{-1,\ldots,-1}_{k},0)$.  Its quiver is
connected because $N=x(J-I)$ has a degree-one arrow between every two
distinct vertices.

Tensor either normalized algebra with a one-vertex factor of profile $(d)$
using the tensor-shift observation above.  Such a factor exists for every
$d\in\mathbb Z$ by the constructions in the proof of
Theorem~\ref{thm:realization}.  This changes only the common shift of the
profile and does not remove the degree-one arrows witnessing connectedness.

After a common integer shift, these constructions therefore give every pair
$(n,d)$ with $n<d$ and every tuple
$(\underbrace{d-1,\ldots,d-1}_{k},d)$ with $d\in\mathbb Z$.

\paragraph{A depth-two example.}
The preceding family gives arbitrarily many negative entries of depth one
sharing a single zero exponent.  The same phenomenon occurs at depth two.
For
\[
  N=\begin{pmatrix}0&x&x\\2x&x^4&2x^2\\2x&2x^2&x^4\end{pmatrix},
\]
the seven nonzero $2\times2$ minors of $I_3+N$ are $(1-x^2)^2$ times
entries of the multiset
\[
  \{1,1,-x,x,-2x,2x,(x^2+1)^2\},
\]
and $\det(I_3+N)=(1-x^2)^4$.  Hence $\delta_1=0$, $\delta_2=2$, and
$\delta_3=4$, so the local Smith exponents of $I_3+N(1+t)$ are
$(0,2,2)$, and $A_N$ has signed profile $(-2,-2,0)$.  Its quiver is
connected, so this is an indecomposable realization with two negative
elementary divisors of depth two and a single zero exponent.

\subsection{Koszul duality and the antipodal point}
\label{sec:koszul-duality}

The local Smith construction at a specified point is the same DVR mechanism
used in \cite[Section~2.1]{kaygun2026localsmithprofiles}.  In that paper it is
applied at $x=1$ to polynomial Calabi--Yau denominators; the surrounding
discussion also recalls that local Smith data are transported under Möbius
changes of variable \cite[Section~1.1]{kaygun2026localsmithprofiles}.  For the
present paper the relevant change is simply $x\mapsto -x$.  Koszul duality
therefore supplies an interpretation of the realization constructions rather
than an additional independent invariant.

Under the standing hypotheses of Section~\ref{sec:local-smith}, assume in
addition that $A$ is Koszul.  Then the minimal resolution of each vertex simple
$S_j$ is linear, with $n$-th term
$\bigoplus_iAe_i\langle n\rangle^{\,b^n_{ij}}$, and write
$B_n=(b^n_{ij})$.  Minimality kills the differentials of
$\Hom(P^{(j)}_\bullet,S_i)$, so
$b^n_{ij}=\dim_\K\Ext^n_A(S_j,S_i)$, where $\Ext$ is computed in the
category of graded modules and summed over all internal shifts;
linearity puts it all in shift $n$.  The Koszul dual
$A^!=\bigoplus_n\Ext^n_A(S,S)$, $S=\bigoplus_jS_j$, graded by
cohomological degree, is then locally finite---each $B_n$ is finite
by $\FP_\infty$---with degree-zero part
$\bigoplus_j\Ext^0(S_j,S_j)\cong E$, and its matrix Hilbert series is
$C_{A^!}(x)=\sum_nB_n^Tx^n$ for the dual idempotents.  (The opposite
convention for the Ext algebra transposes the matrices; the local Smith
exponents are insensitive to this.)  By the completed Euler construction of
Section~\ref{sec:euler-inversion} and rational algebraization, linearity of
the resolutions gives
$\mathcal E_A(x)=\sum_n(-1)^nB_nx^n$, whence the matrix form of the
classical Koszul numerical identity.  Although $A$ and $A^!$ are algebras over
$\K$, the matrices $B_n$ have integral entries; under the standing rationality
hypotheses the following identity therefore lies in $M_r(\mathbb Q(x))$:
\begin{equation}
\label{eq:koszul-numerics}
  \mathcal E_A(x)=C_{A^!}(-x)^T.
\end{equation}

By \eqref{eq:koszul-numerics}, $C_{A^!}$ and
$\mathcal E_{A^!}=C_{A^!}^{-1}$ are rational.  We use
$\mathcal E_{A^!}$ here only as the inverse matrix Hilbert series; the Koszul
dual need not itself satisfy the vertex-simple $\FP_\infty$ hypothesis required
for the completed Euler interpretation.  Let $c_1\leq\cdots\leq c_r$ be the signed local exponents of
$\mathcal E_{A^!}$ at $x=-1$.  No separate structural result is needed here:
the local Smith formalism of \cite[Section~2.1]{kaygun2026localsmithprofiles},
together with the standard transport under a change of local parameter
discussed in \cite[Section~1.1]{kaygun2026localsmithprofiles}, immediately
gives
\[
 (a_1,\ldots,a_r)=(-c_r,\ldots,-c_1).
\]
Indeed, substituting $x=1+t$ in \eqref{eq:koszul-numerics} gives
$\mathcal E_A(1+t)=C_{A^!}(-1-t)^T$.  The right side is the value of
$C_{A^!}$ at the point $-1$ in the local parameter $s=-t$, and neither
the transpose nor the sign of the parameter changes Smith exponents.
So the profile of $A$ at $1$ equals the exponent multiset of
$C_{A^!}=\mathcal E_{A^!}^{-1}$ at $-1$, which is the negated, reversed
exponent multiset of $\mathcal E_{A^!}$ at $-1$.

The negative seed of Proposition~\ref{prop:negative-seed} gives a concrete
rank-one check of the antipodal formula.  The syzygy recursion in its proof is
linear, so $R$ is Koszul, and its quadratic dual is
\[
 R^!=\K\langle u,v,w\rangle/(u^2,uv,vu,v^2).
\]
A direct word count gives
\[
 C_{R^!}(x)=\frac{1+2x}{(1+x)(1-2x)}.
\]
Thus $C_{R^!}$ has a simple pole at $x=-1$, equivalently
$\mathcal E_{R^!}$ has local exponent $1$ there, and
the preceding antipodal identity returns the profile $(-1)$ of $R$.

The involution $x\mapsto-x$ of the projective line exchanges the
points $1$ and $-1$ and fixes $0$ and $\infty$; at the level of
lattice data, \eqref{eq:koszul-numerics} says that the lattice of $A$
at $1$ and the lattice of $A^!$ at $-1$ are transpose--inverse to one
another.  The dual constructions of
this section realize the duality exactly.  For
$N=xM$ with $M\in M_r(\mathbb Z_{\geq0})$, the square-zero algebra
$A_N$ is quadratic with relation space all of $V\otimes_EV$, so its
quadratic dual is the tensor algebra on the dual bimodule, and
\[
  \mathcal E_{A_N}(x)=(I_r+xM)^{-1}
  =\bigl((I_r+xM^T)^{-1}\bigr)^T
  =C_{A_N^!}(-x)^T.
\]
The antipodal identity therefore closes the circle with the realization
constructions of this section: tensor and square-zero algebras
exchange the local elementary-divisor data exactly as Koszul duality predicts.

\section{Graded growth}
\label{sec:growth}

The profile controls the pole order at $x=1$, but the other poles decide
whether that local order is a polynomial-growth degree at all.  Euler
numerators connect the local Smith torsion to the pole at $1$; rational growth
then separates the finite and exponential branches.

\subsection{Euler numerators and local pole order}

For a finitely generated graded projective $P$, write, uniquely as in
Lemma~\ref{lem:graded-projective-decomposition},
\[
  P\cong\bigoplus_{j=1}^r Ae_j\otimes_\K W_j
\]
and put $q_P(x)_j=\Hilb_x(W_j)\in\mathbb Z[x,x^{-1}]$.  If $M$ is a
perfect graded module and $P_\bullet\twoheadrightarrow M$ is a finite
projective resolution, define its \emph{Euler numerator}
\[
q_M(x):=\sum_n(-1)^nq_{P_n}(x)
  \in\mathbb Z[x,x^{-1}]^r.
\]
Additivity of Hilbert series gives
\begin{equation}
\label{eq:module-hilbert-euler}
  \Hilb_M(x)=C_A(x)q_M(x)=\mathcal E_A(x)^{-1}q_M(x).
\end{equation}
Since $\mathcal E_A$ is invertible over $\mathbb Q(x)$, this identity also shows
that $q_M$ is independent of the chosen finite projective resolution.  

Let $\pi_A^+:\mathbb Q\llbracket t\rrbracket^r\twoheadrightarrow\mathcal G_A^+$
be the quotient map.  For $\xi\in\mathcal G_A^+$ and
$h\in\mathbb Q((t))^r$, put
\[
  \lambda_A^+(\xi):=\min\{d\geq0:t^d\xi=0\},
  \qquad
  \pole_1(h):=\min\{d\geq0:t^dh\in\mathbb Q\llbracket t\rrbracket^r\}.
\]

\begin{proposition}[Local Euler--Smith pole formula]
\label{prop:local-growth-class}
For every perfect graded module $M$,
\[
 \lambda_A^+(\pi_A^+(q_M(1+t)))=\pole_1(\Hilb_M).
\]
\end{proposition}

\begin{proof}
More generally, for $q\in\mathbb Q\llbracket t\rrbracket^r$, write
$\widetilde q=Uq$ in \eqref{eq:local-smith-form}.  Then both the torsion
order of $\pi_A^+(q)$ and the pole order of $C_A(1+t)q$ equal
\[
 \max_i\bigl(a_i-\ord_t(\widetilde q_i)\bigr)_+,
\]
because
$C_A(1+t)=V\operatorname{diag}(t^{-a_i})U$ and
$V\in GL_r(\mathbb Q\llbracket t\rrbracket)$ preserves every lattice
$t^{-d}\mathbb Q\llbracket t\rrbracket^r$.  Apply this to
$q=q_M(1+t)$ and use \eqref{eq:module-hilbert-euler}.
\end{proof}

\subsection{The global--local growth dichotomy}

The standing $\FP_\infty$ hypothesis implies that $A$ is finitely
generated.  Indeed, $A/J$ is of type $\FP_1$, so $J$ is finitely generated as
a graded left ideal; homogeneous generators of $J$, together with $E$, generate
$A$ by induction on degree.  This is the implication used in
\cite[Lemma~2.3]{ReyesRogalski2022}.

For a nonzero Laurent series
$h(x)=\sum_{n\geq n_0}c_nx^n$ with $c_n\geq0$, write
\[
  \GKdim h:=\limsup_{N\to\infty}
  \frac{\log(\sum_{n_0\leq n\leq N}c_n)}{\log N}.
\]
We say that $h$ has \emph{exponential growth} if
\[
  \limsup_{N\to\infty}
  \left(\sum_{n_0\leq n\leq N}c_n\right)^{1/N}>1.
\]
For a nonzero vector series with nonnegative coefficients, use the same
definitions after summing the coordinates.  Since $A$ is finitely generated, for a finitely
generated graded module, degree truncation and the usual word-length filtration
are comparable up to linear rescaling; hence these definitions agree with the
ordinary notions of GK dimension and exponential growth.  Compare the
discussion of growth of graded algebras in
\cite[Section~2]{ReyesRogalski2019}.

The following elementary normalization verifies the denominator hypothesis in
the scalar result we use.

\begin{lemma}[Integral denominator normalization]
\label{lem:integral-denominator}
Let $0\neq h\in\mathbb Z((x))\cap\mathbb Q(x)$.  Then
\[
  h(x)=\frac{p(x)}{q(x)}
\]
with relatively prime $p\in\mathbb Z[x,x^{-1}]$ and
$q\in\mathbb Z[x]$ satisfying $q(0)=\pm1$.
\end{lemma}

\begin{proof}
Multiplying by a power of $x$, it is enough to treat
$h\in\mathbb Z\llbracket x\rrbracket$.  Choose a reduced presentation
$h=p/q$ with $q\in\mathbb Z[x]$ primitive and $p,q$ relatively prime in
$\mathbb Q[x]$.  Since $p=qh$ and both $q$ and $h$ have integral coefficients,
$p\in\mathbb Z[x]$.

Clearing denominators in a B\'ezout relation gives
$a,b\in\mathbb Z[x]$ and $0\neq m\in\mathbb Z$ with
$ap+bq=m$.  Hence $g:=ah+b\in\mathbb Z\llbracket x\rrbracket$ satisfies
$qg=m$.  Write $g=d g_0$, where $d>0$ is the gcd of the coefficients of $g$;
then $g_0$ is primitive.  The product of a primitive polynomial and a primitive
integral power series is primitive: modulo every prime both factors remain
nonzero, and $\mathbb F_p\llbracket x\rrbracket$ is a domain.  Thus
$qg_0=m/d$ is primitive.  Since the left side has integral coefficients,
$m/d\in\mathbb Z$; being a primitive constant, it is therefore $\pm1$.
Thus $qg_0=\pm1$, and comparison of constant terms gives $q(0)=\pm1$.
Multiplying back by the initial power of $x$ changes only the Laurent numerator;
since $q(0)=\pm1$, no common factor is introduced.
\end{proof}

By Lemma~\ref{lem:integral-denominator}, every nonzero rational Laurent series
with integral coefficients can be put in the precise form required by
Reyes--Rogalski \cite[Lemma~2.7]{ReyesRogalski2019}.  We therefore use
\cite[Lemma~2.7(1)--(2)]{ReyesRogalski2019}: for a rational Laurent series with
nonnegative integral coefficients, finite GK dimension is equivalent to all
nonzero poles being roots of unity; in that case the GK dimension is the pole
order at $x=1$, and otherwise the growth is exponential.

For a vector rational series $\mathbf h=(h_1,\ldots,h_r)^T$ with
nonnegative coefficients,
\begin{equation}
\label{eq:vector-gk-max}
  \GKdim \mathbf h=\max_{i:h_i\neq0}\GKdim h_i
\end{equation}
whenever the right-hand side is finite, while one exponentially growing
coordinate makes $\mathbf h$ exponentially growing.  This follows by comparing
the total partial sums with the largest coordinatewise partial sum.

\begin{theorem}[Global--local growth dichotomy for perfect modules]
\label{thm:perfect-global-local-dichotomy}
Let $M$ be a nonzero perfect graded module.  Exactly one of the following occurs.
\begin{enumerate}[(i)]
\item Every nonzero pole of every nonzero coordinate of $\Hilb_M$ is a root
of unity.  Then $M$ has finite GK dimension and
\begin{equation}
\label{eq:perfect-finite-local-degree}
  \GKdim M
  =\pole_1(\Hilb_M)
  =\lambda_A^+(\pi_A^+(q_M(1+t))).
\end{equation}
\item Some coordinate of $\Hilb_M$ has a nonzero pole that is not a root of
unity.  Then $M$ has exponential growth.
\end{enumerate}
\end{theorem}

\begin{proof}
Each coordinate of $\Hilb_M=\mathcal E_A^{-1}q_M$ is a rational Laurent series with
nonnegative integral coefficients.  Apply the Reyes--Rogalski scalar theorem quoted above coordinatewise.  If one coordinate
has exponential growth, then so does their sum and hence $M$.  Otherwise every
coordinate has finite GK dimension, and \eqref{eq:vector-gk-max} identifies
$\GKdim M$ with the largest pole order at $1$, namely $\pole_1(\Hilb_M)$.
Proposition~\ref{prop:local-growth-class} gives the last equality in
\eqref{eq:perfect-finite-local-degree}.
\end{proof}

The theorem deliberately separates the two roles: the global pole set
chooses the branch, while the positive stalk at $1$ computes the degree in the
finite branch.  The next example shows why both pieces are necessary.

\paragraph{The local stalk does not detect the growth branch.}
Let $A=\K\langle X,Y\rangle$ with $\deg X=\deg Y=1$.  Then
$C_A(x)=(1-2x)^{-1}$ and $\mathcal E_A(x)=1-2x$, so
$\mathcal E_A(1+t)=-1-2t$ is a unit and $\mathcal G_A^+=0$.  Nevertheless
both finite and exponential growth occur among perfect graded modules.  The
augmentation module $\K$ has the finite resolution
\[
  0\longrightarrow A\langle1\rangle^2
  \longrightarrow A\longrightarrow\K\longrightarrow0
\]
and GK dimension zero.  On the other hand, right multiplication by $Y$ gives an
exact sequence of left modules
\[
  0\longrightarrow A\langle1\rangle
  \xrightarrow{\,\cdot Y\,}A
  \longrightarrow A/AY\longrightarrow0,
\]
so $A/AY$ is perfect, while
\[
  \Hilb_{A/AY}(x)=\frac{1-x}{1-2x}
\]
has a pole at $x=1/2$ and exponential growth.  Both modules have local
Euler--Smith degree zero.  Thus the stalk at $1$ measures the local pole degree
but cannot decide the global growth branch.

Applying the theorem to the indecomposable projectives gives an algebra-level
form that no longer requires finite GK dimension as an input.

\begin{corollary}[Growth dichotomy for the algebra]
\label{cor:euler-smith-growth}
Exactly one of the following occurs.
\begin{enumerate}[(i)]
\item Every nonzero pole of every entry of $C_A(x)$ is a root of unity.  Then
$\GKdim A<\infty$ and
\[
  \GKdim A=\Nil_t(\mathcal G_A^+)=\max\{0,a_r\}.
\]
\item Some entry of $C_A(x)$ has a nonzero pole that is not a root of unity.
Then $A$ has exponential growth.
\end{enumerate}
\end{corollary}

\begin{proof}
For $P_j=Ae_j$, the vector Hilbert series $\Hilb_{P_j}$ is the $j$th column of
$C_A$.  Since $A=\bigoplus_jP_j$, the growth of $A$ is the maximum of the
growth of these finitely many projectives.  Theorem~\ref{thm:perfect-global-local-dichotomy}
therefore gives the dichotomy.  In the finite branch $q_{P_j}=e_j$, so
\eqref{eq:perfect-finite-local-degree} gives
$\GKdim(P_j)=\lambda_A^+(\pi_A^+(e_j))$.  The elements
$\pi_A^+(e_j)$ generate $\mathcal G_A^+$ as a
$\mathbb Q\llbracket t\rrbracket$-module, hence the maximum
of their torsion orders is $\Nil_t(\mathcal G_A^+)$.  The Smith decomposition
identifies this nilpotence index with $\max\{0,a_r\}$.
\end{proof}

\begin{example}[A cyclic Euler--Smith module]
\label{ex:cyclic-growth}
Let $Q$ be obtained from the oriented chain $1\to2\to\cdots\to r$ by adjoining
one loop at every vertex, and put $A=\K Q$ with the path-length grading.  We
use the convention that an arrow $i\to j$ lies in $e_jAe_i$, so that
$C_A(x)_{ji}$ counts paths from $i$ to $j$; this convention is kept in
Examples~\ref{ex:graded-reflection} and~\ref{ex:twist-sensitivity}.  If
$N$ is the incidence matrix of the chain, $N_{j+1,j}=1$, the standard tensor-algebra exact sequence gives
\[
  \mathcal E_A(x)=I_r-x(I_r+N).
\]
Since $N$ is a single nilpotent Jordan block, the Smith form of
$tI_r+(1+t)N$ is $\operatorname{diag}(1,\ldots,1,t^r)$, so
\[
  \mathcal G_A^+\cong\mathbb Q\llbracket t\rrbracket/(t^r).
\]
The column relations give $te_j\equiv-(1+t)e_{j+1}$ for $j<r$ and
$te_r\equiv0$, hence
\[
  \lambda_A^+(\pi_A^+(e_j))=r-j+1=\GKdim(Ae_j).
\]
Thus one cyclic Smith summand records all vertex-projective growth degrees
$1,\ldots,r$, not merely their maximum.
\end{example}

\subsection{Wholly negative profiles force exponential corners}

\begin{proposition}[A negative top exponent forces exponential corners]
\label{prop:negative-top}
If $a_r<0$, then for every $i$ the corner Hilbert series
$\Hilb_x(e_iAe_i)$ has radius of convergence strictly less than $1$.
\end{proposition}

\begin{proof}
If $a_r<0$, then
$C_A(1+t)=V\operatorname{diag}(t^{-a_i})U$ has entries in
$t\mathbb Q\llbracket t\rrbracket$, so every entry of $C_A$ is regular at $x=1$ and vanishes there.
Let $h=\Hilb_x(e_iAe_i)$, a rational function whose series has
nonnegative integer coefficients and constant term $1$, and let $\rho$
be its radius of convergence.  If $\rho>1$, then $h(1)$ equals the sum
of the coefficients, which is at least $1$, contradicting $h(1)=0$.  If
$\rho=1$, then $h$ has a pole $\zeta$ with $|\zeta|=1$; from
$|h(s\zeta)|\leq h(s)$ for $0<s<1$ and $|h(s\zeta)|\to\infty$ as
$s\to1^-$, we get $h(s)\to\infty$, contradicting regularity of $h$ at
$1$.  Hence $\rho<1$, so the coefficient sums of $h$ grow
exponentially.  Finally, the growth of $A$ dominates the growth of
every corner.
\end{proof}

Consequently every corner $e_iAe_i$ has exponential graded coefficient
growth, and so does $A$.  This is a constraint on growth, not on
realizability: Theorem~\ref{thm:realization} shows that wholly negative
profiles occur freely, while Proposition~\ref{prop:negative-top} places every
algebra realizing one in the exponential branch of
Corollary~\ref{cor:euler-smith-growth}.  For the seed algebra
$R$ of Proposition~\ref{prop:negative-seed}, the unique corner has radius of
convergence $1/2$, exactly as the proposition predicts.

\section{Categorical invariance}
\label{sec:morita}

Having established the growth meaning and range of the signed profile, we next
ask whether it depends on a presentation of the algebra.  The answer is governed
by compatibility with the grading shift.  A shift-compatible graded Morita
equivalence gives monomial transport of the matrix Hilbert series, while a
shift-compatible equivalence of perfect derived categories replaces monomial
transport by a Laurent-polynomial congruence.  The Zhang-twist example at the end
of the section shows that an equivalence of graded module categories which does
not respect the translation is too coarse.

\subsection{Shift-compatible equivalences}

Let $A$ satisfy the standing hypotheses of Section~\ref{sec:local-smith}, and
let $B$ be a locally finite positively graded algebra with elementary
degree-zero part, that is, $B_0=\bigoplus_{j=1}^s\K f_j$ for primitive
orthogonal idempotents $f_1,\ldots,f_s$.  Write
$A_0=\bigoplus_{i=1}^r\K e_i$ as before, and let
$\operatorname{GrMod}$ denote the category of graded left modules with
degree-preserving morphisms.  For a graded Morita equivalence the proof below
shows that the standing hypotheses pass from $A$ to $B$.  A \emph{graded Morita equivalence} is a
$\K$-linear equivalence $F:\operatorname{GrMod}A\to\operatorname{GrMod}B$ together with
a natural isomorphism
$F\circ\langle1\rangle\cong\langle1\rangle\circ F$; iterating, $F$ commutes
with every shift $\langle n\rangle$.

\begin{lemma}[Small projectives]
\label{lem:small-projectives}
A graded projective $A$-module $P$ is finitely generated if and only if
$\Hom_{\operatorname{GrMod}}(P,-)$ preserves arbitrary direct sums.
Moreover, the finitely generated indecomposable graded projectives are
exactly the modules $Ae_i\langle m\rangle$, and these are pairwise
nonisomorphic.
\end{lemma}

\begin{proof}
If $P$ has finitely many homogeneous generators, every degree-zero
morphism from $P$ to a direct sum lands in finitely many summands.
Conversely, choosing homogeneous generators gives a split surjection
$\bigoplus_{\alpha\in I}A\langle m_\alpha\rangle\twoheadrightarrow P$; if
$\Hom(P,-)$ preserves direct sums, the resulting section of $P$ factors
through a finite subsum, so $P$ is a direct summand of a finitely
generated graded free module.

By Lemma~\ref{lem:graded-projective-decomposition}, a finitely generated
graded projective is a finite direct sum of shifted modules
$Ae_i\langle m\rangle$; each of these is indecomposable because
$\operatorname{End}_{\operatorname{GrMod}}(Ae_i\langle m\rangle)
\cong(e_iAe_i)_0=\K$ is local.  Finally, $Ae_i\langle m\rangle$ has top
$S_i\langle m\rangle$, and the modules $S_i\langle m\rangle$ are pairwise
nonisomorphic, so the same holds for their projective covers.
\end{proof}

Let $F:\operatorname{GrMod}A\to\operatorname{GrMod}B$ be a graded Morita
equivalence.  By Lemma~\ref{lem:small-projectives}, applied also to a
quasi-inverse, $r=s$ and there are a permutation $\sigma$ and integers
$m_1,\ldots,m_r$ such that
$F(Ae_i)\cong Bf_{\sigma(i)}\langle m_i\rangle$.  Let $\Pi$ be the
permutation matrix of $\sigma$ and put
$X(x)=\operatorname{diag}(x^{m_1},\ldots,x^{m_r})$.

\begin{theorem}[Morita invariance of the signed profile]
\label{thm:morita-invariance}
With the notation above,
\begin{equation}
\label{eq:morita-transport}
  C_B(x)=\Pi\,X(x)\,C_A(x)\,X(x)^{-1}\,\Pi^T.
\end{equation}
Hence $A$ and $B$ have the same signed profile and
$\mathcal G_B^\pm\cong\mathcal G_A^\pm$.
\end{theorem}

\begin{proof}
With the shift convention of Section~\ref{sec:euler-inversion},
$(Ae_j\langle-n\rangle)_0=(Ae_j)_n$.  With the convention that
$A_n=B_n=0$ for $n<0$, for every $n\in\mathbb Z$ we therefore have
\[
\begin{aligned}
  e_iA_ne_j
  &\cong\Hom_{\operatorname{GrMod}A}(Ae_i,Ae_j\langle-n\rangle)\\
  &\cong\Hom_{\operatorname{GrMod}B}
    \bigl(Bf_{\sigma(i)}\langle m_i\rangle,
          Bf_{\sigma(j)}\langle m_j-n\rangle\bigr)\\
  &\cong\bigl(f_{\sigma(i)}Bf_{\sigma(j)}\bigr)_{n+m_i-m_j}.
\end{aligned}
\]
Summing over $n$ gives
$C_A(x)_{ij}=x^{m_j-m_i}C_B(x)_{\sigma(i)\sigma(j)}$, which is
\eqref{eq:morita-transport}.  In particular $C_B$ is rational.  Moreover,
the equivalence and a quasi-inverse preserve finitely generated projectives,
simple objects, and projective covers.  Applying them to minimal resolutions
therefore transports the graded $\FP_\infty$ property between the vertex
simples of $A$ and $B$.

Finally, the transport matrix has entries in
$\mathbb Z[x,x^{-1}]$, so
$\mathcal E_B(1+t)=W\mathcal E_A(1+t)W^{-1}$ with
$W=\Pi\,X(1+t)\in GL_r(\mathbb Q\llbracket t\rrbracket)$.  Thus
Lemma~\ref{lem:gauge} identifies the canonical signed profiles and the modules
$\mathcal G^\pm$.
\end{proof}

\subsection{Derived invariance}

For perfect derived categories the relevant change of basis lives on graded
$K$-theory, where the matrix Hilbert series is the Gram matrix of an Euler
form.  Let $\operatorname{per}A$ denote the homotopy category of bounded complexes
of finitely generated graded projective left $A$-modules, a
triangulated category on which the shift $\langle1\rangle$ acts.  We regard
$K_0(\operatorname{per}A)$ as a $\mathbb Z[x,x^{-1}]$-module by
$x[P]=[P\langle1\rangle]$.  For graded Grothendieck groups as organizing invariants, see
\cite{Hazrat2016}.

For $P,Q\in\operatorname{per}A$ put
\[
  \langle P,Q\rangle
  :=\sum_{n\in\mathbb Z}
    \Bigl(\sum_{i\in\mathbb Z}(-1)^i
    \dim_\K\operatorname{Hom}(P,Q\langle-n\rangle[i])\Bigr)x^n.
\]

\begin{lemma}[The Euler form on graded K-theory]
\label{lem:euler-form}
\begin{enumerate}[(i)]
\item The classes $[Ae_i]$ form a $\mathbb Z[x,x^{-1}]$-basis.
\item The pairing above descends to a sesquilinear form with Gram matrix $C_A$.
\end{enumerate}
\end{lemma}

\begin{proof}
(i) The Euler characteristic
$[P^\bullet]\mapsto\sum_i(-1)^i[P^i]$ is additive on mapping cones,
hence defines a map from $K_0(\operatorname{per}A)$ to the split
Grothendieck group of the category of finitely generated graded
projectives; the map $[P]\mapsto[P[0]]$ goes the other way, one
composite is the identity by inspection, and the other is the identity
by induction on the length of a complex, using the brutal truncation
triangles.  By Lemma~\ref{lem:graded-projective-decomposition} the
split Grothendieck group is free abelian on the classes
$[Ae_i\langle n\rangle]$, which is the statement.

(ii) For $P,Q$ bounded complexes of finitely generated graded
projectives, each $\operatorname{Hom}(P,Q\langle-n\rangle[i])$ is a
subquotient of a finite sum of spaces
$\operatorname{Hom}(Ae_a\langle m\rangle,Ae_b\langle m'\rangle)$,
which are finite-dimensional by local finiteness and vanish for
$n\ll0$ by positivity; the inner sum is finite because the complexes
are bounded.  So the series lies in $\mathbb Z((x))$.  Fixing one
variable, a triangle in the other induces a long exact sequence of
finite-dimensional Hom-spaces, so the Euler characteristics are
additive and the pairing descends to $K_0$.  The identities
$\langle P\langle1\rangle,Q\rangle=x^{-1}\langle P,Q\rangle$ and
$\langle P,Q\langle1\rangle\rangle=x\langle P,Q\rangle$ are immediate
from the definition.  Finally, a degree-preserving map
$Ae_i\to Ae_j\langle-n\rangle$ is right multiplication by an element
of $(e_iAe_j)_n$, higher $\operatorname{Hom}$'s vanish, and summing
gives $\langle[Ae_i],[Ae_j]\rangle=\Hilb_x(e_iAe_j)=(C_A)_{ij}$.
\end{proof}

Call a triangulated equivalence
$\Phi:\operatorname{per}A\to\operatorname{per}B$ \emph{shift-compatible} if
$\Phi\circ\langle1\rangle\cong\langle1\rangle\circ\Phi$ naturally.

\begin{theorem}[Derived invariance]
\label{thm:derived-invariance}
Assume now that both $A$ and $B$ satisfy the standing hypotheses of
Section~\ref{sec:local-smith}.  Let
$\Phi:\operatorname{per}A\to\operatorname{per}B$ be a $\K$-linear
shift-compatible equivalence.  Then $r=s$ and there is
$S\in GL_r(\mathbb Z[x,x^{-1}])$ such that
\begin{equation}
\label{eq:derived-transport}
  C_A(x)=S(x^{-1})^T\,C_B(x)\,S(x).
\end{equation}
Consequently $A$ and $B$ have the same signed Euler--Smith profile at $x=1$
and $\mathcal G_A^\pm\cong\mathcal G_B^\pm$.
\end{theorem}

\begin{proof}
Being triangulated, $\Phi$ induces a homomorphism
$\varphi:K_0(\operatorname{per}A)\to K_0(\operatorname{per}B)$, which
is $\mathbb Z[x,x^{-1}]$-linear because $\Phi$ commutes with $\langle1\rangle$
and bijective because $\Phi$ is an equivalence.  Lemma~\ref{lem:euler-form}
identifies these groups as free $\mathbb Z[x,x^{-1}]$-modules of ranks $r$ and
$s$, so
$r=s$; writing
$\varphi[Ae_j]=\sum_kS_{kj}[Bf_k]$ defines $S\in GL_r(\mathbb Z[x,x^{-1}])$.  The same lemma identifies the Euler pairings with Gram matrices $C_A$ and
$C_B$.  Since an equivalence commuting with $\langle1\rangle$ preserves
the Hom-spaces defining the pairing,
$\langle\varphi(u),\varphi(v)\rangle_B=\langle u,v\rangle_A$ for all
$u,v$; evaluating on basis vectors and expanding by sesquilinearity gives
\eqref{eq:derived-transport}.

At $x=1$ the matrices $S(x)$ and $S(x^{-1})^T$ are invertible over
$\mathbb Q\llbracket t\rrbracket$, since their determinants are Laurent
monomials up to sign.  Inverting \eqref{eq:derived-transport} gives
$\mathcal E_B=S(x)\mathcal E_A S(x^{-1})^T$, so Lemma~\ref{lem:gauge}
identifies the signed profiles and Euler--Smith modules.
\end{proof}

Thus the signed profile is an invariant of the shift-equipped perfect category.
A graded Morita equivalence is the special case $S=\Pi X$ of
\eqref{eq:derived-transport}; thus derived invariance replaces monomial
transport by an arbitrary Laurent-polynomial change of basis.

\begin{example}[A genuinely derived change of basis]
\label{ex:graded-reflection}
Let $A$ be the path algebra of $1\to2\to3$ with arrows of degree $1$, so
\[
  C_A=\begin{pmatrix}1&0&0\\x&1&0\\x^2&x&1\end{pmatrix}.
\]
The resolution $0\to Ae_3\langle1\rangle\to Ae_2\to S_2\to0$ shows that,
for $d\geq1$,
$T=Ae_1\oplus Ae_2\oplus S_2\langle d\rangle$ is a graded tilting object:
the resolution gives
$\operatorname{Hom}_{\operatorname{per}A}(T,T\langle n\rangle[i])=0$
for every $n$ and every $i\neq0$, and its shifted triangle recovers
$Ae_3\langle d+1\rangle$ from the shift-stable thick subcategory generated by
$T$.  The standard tilting construction \cite{Rickard1989}, in its graded
form, therefore gives a shift-compatible equivalence
$\operatorname{per}B\to\operatorname{per}A$ sending the three indecomposable
projectives of $B$ to $Ae_1$, $Ae_2$ and $S_2\langle d\rangle$; a direct
endomorphism calculation identifies $B$ with the path algebra of
$1\to2\leftarrow3$, where the reflected arrow has degree $d$.  On graded
$K_0$ this equivalence has change of basis
\[
  S=\begin{pmatrix}1&0&0\\0&1&x^d\\0&0&-x^{d+1}\end{pmatrix},
\]
the third column recording $[S_2\langle d\rangle]=x^d[Ae_2]-x^{d+1}[Ae_3]$,
and \eqref{eq:derived-transport}, with the roles of $A$ and $B$ interchanged,
reads
\[
  S(x^{-1})^T C_A S(x)
  =\begin{pmatrix}1&0&0\\x&1&x^d\\0&0&1\end{pmatrix}=C_B.
\]
This change of basis is not monomial, so the equivalence is not a graded
Morita equivalence; it exhibits the genuinely larger scope of
Theorem~\ref{thm:derived-invariance}.
\end{example}

\subsection{Sharpness: Zhang twists}

A degree-preserving automorphism $\tau$ defines the Zhang twist $A^\tau$,
and Zhang's construction gives an equivalence of graded module categories
\cite[Theorem~3.1]{Zhang1996}.  Such an equivalence need not commute with the
grading shift, and the profile can change.

\begin{example}[Shift compatibility is necessary]
\label{ex:twist-sensitivity}
Let $A=\K[X]\times\K[X]$ with $\deg X=1$, and let $\tau$ exchange the two
factors.  If $\pi=(1\,2)$, then
$e_i(A^\tau)_ne_j=e_iA_ne_{\pi^n(j)}$, hence
\[
  C_{A^\tau}(x)=\frac1{1-x^2}\begin{pmatrix}1&x\\x&1\end{pmatrix},
  \qquad
  \mathcal E_{A^\tau}(x)=\begin{pmatrix}1&-x\\-x&1\end{pmatrix}.
\]
The degree-one generators identify $A^\tau$ with the tensor algebra on two
opposite arrows, so it satisfies the standing hypotheses.  At $x=1$ the
profiles are $(1,1)$ for $A$ and $(0,1)$ for $A^\tau$.  Thus the two algebras
have the same total Hilbert series, the same GK dimension, and equivalent
graded module categories, but different Euler--Smith profiles.  The shift
compatibility in Theorems~\ref{thm:morita-invariance} and
\ref{thm:derived-invariance} is therefore essential.
\end{example}

\bibliographystyle{amsalpha}
\bibliography{references}

@article{AvramovBuchweitzSally1997,
  author  = {Avramov, L. L. and Buchweitz, R.-O. and Sally, J. D.},
  title   = {Laurent coefficients and {Ext} of finite graded modules},
  journal = {Math. Ann.},
  volume  = {307},
  number  = {3},
  year    = {1997},
  pages   = {401--415},
  doi     = {10.1007/s002080050041},
}

@book{Hazrat2016,
  author    = {Hazrat, R.},
  title     = {Graded Rings and Graded {Grothendieck} Groups},
  series    = {London Mathematical Society Lecture Note Series},
  volume    = {435},
  publisher = {Cambridge University Press},
  address   = {Cambridge},
  year      = {2016},
  doi       = {10.1017/CBO9781316717134},
}

@article{ReyesRogalski2019,
  author  = {Reyes, M. L. and Rogalski, D.},
  title   = {Growth of graded twisted {Calabi--Yau} algebras},
  journal = {J. Algebra},
  volume  = {539},
  year    = {2019},
  pages   = {201--259},
  doi     = {10.1016/j.jalgebra.2019.07.029},
}

@article{ReyesRogalski2022,
  author  = {Reyes, M. L. and Rogalski, D.},
  title   = {Graded twisted {Calabi--Yau} algebras are generalized {Artin--Schelter} regular},
  journal = {Nagoya Math. J.},
  volume  = {245},
  year    = {2022},
  pages   = {100--153},
  doi     = {10.1017/nmj.2020.32},
}

@article{WilkeningYu2011,
  author  = {Wilkening, J. and Yu, J.},
  title   = {A local construction of the {Smith} normal form of a matrix polynomial},
  journal = {J. Symbolic Comput.},
  volume  = {46},
  number  = {1},
  year    = {2011},
  pages   = {1--22},
  doi     = {10.1016/j.jsc.2010.06.025},
}

@article{Zhang1996,
  author  = {Zhang, J. J.},
  title   = {Twisted graded algebras and equivalences of graded categories},
  journal = {Proc. London Math. Soc. (3)},
  volume  = {72},
  number  = {2},
  year    = {1996},
  pages   = {281--311},
}

@misc{kaygun2026localsmithprofiles,
  author = {Kaygun, Atabey},
  title = {Local {Smith} Profiles of Twisted {Calabi}--{Yau} Algebras},
  year = {2026},
  howpublished = {Preprint, {arXiv}:2608.14057 [math.{RA}]},
  url = {https://arxiv.org/abs/2608.14057},
  arXiv = {arXiv:2608.14057}
}

@article{Rickard1989,
  author  = {Rickard, J.},
  title   = {Morita theory for derived categories},
  journal = {J. London Math. Soc. (2)},
  volume  = {39},
  number  = {3},
  year    = {1989},
  pages   = {436--456},
  doi     = {10.1112/jlms/s2-39.3.436},
}

\end{document}